\documentclass{ws-ijnt}
\let\trimmarks\relax
\usepackage{amsmath,amssymb}
\usepackage{url}
\usepackage[utf8x]{inputenc}
\usepackage[T1]{fontenc}

\newcommand     {\R}            {{\mathbb R}}
\newcommand     {\C}            {{\mathbb C}}

\newcommand     {\SL}[1][\Z]    {\sym{SL}(2,{#1})}
\newcommand     {\Z}            {{\mathbb Z}}

\newcommand     {\sym}[1]       {\operatorname{#1}}

\newcommand     {\abs}[1]       {{\left\lvert{#1}\right\rvert}}
\newcommand     {\kro}[2]       {{\left(\frac{#1}{#2}\right)}}

\newcommand{\GL}[2]{GL(\ensuremath{#1},\ensuremath{#2})}
\newcommand{\Sp}[2]{Sp(\ensuremath{#1},\ensuremath{#2})}
\newcommand{\GSp}[2]{GSp^+(\ensuremath{#1},\ensuremath{#2})}

\renewcommand{\H}{\mathcal{H}}

\newcommand{\para}[1]{\Gamma^{\text{para}}[#1]}
\DeclareMathOperator{\disc}{\text{disc}}

\title{%
  A B\"ocherer-Type Conjecture for Paramodular Forms
}

\author{%
Nathan C. Ryan, Gonzalo Tornar\'ia
\thanks{This project was supported by the National Science Foundation
under FRG Grant No. DMS-0757627 and the Fulbright Commission in
Uruguay. The computations were carried out in a computer system funded
by PEDECIBA.}
}

\date{%
  \today}

\address{
{\parskip 0pt
Department of Mathematics, Bucknell University\endgraf
nathan.ryanATbucknell.edu\endgraf
Centro de Matem\'atica, Universidad de la Rep\'ublica\endgraf
tornariaATcmat.edu.uy\endgraf
}
}

\begin{document}

\maketitle

\begin{abstract}
In the 1980s Böcherer formulated a conjecture relating the central value of the quadratic twists of the spinor $L$-function attached to a Siegel modular form $F$ to the coefficients of $F$. He proved the conjecture when $F$ is a Saito-Kurokawa lift.  Later Kohnen and Kuss gave numerical evidence for the conjecture in the case when $F$ is a rational eigenform that is not a Saito-Kurokawa lift.  In this paper we develop a conjecture relating the central value of the quadratic twists of the spinor $L$-function attached to a paramodular form and the coefficients of the form.  We prove the conjecture in the case when the form is a Gritsenko lift and provide numerical evidence when it is not a lift.
\end{abstract}

\section{Introduction}\label{sec:intro}

Conjectures about central values of $L$-functions abound; for example the Conjecture of Birch and Swinnerton-Dyer predicts that the order of vanishing of the central critical values of an elliptic curve $L$-functions is equal to the rank of the elliptic curve's Mordell-Weil group.  Finding an asymptotic formula fo $V_E(x)$, the number of quadratic twists by discriminant $d\leq x$ of an elliptic curve $E$ with vanishing central value, might provide some insight into how often the twists of $E$ have rank at least 2.  

Sarnak has predicted the asymptotic size of this count based on
formulas of Waldspurger \cite{Waldspurger} and Kohnen-Zagier \cite{KohnenZagier}
which relate the central value
of the quadratic twists of an elliptic curve $L$-function to squares
of Fourier coefficients of half-integral weight modular forms.
These formulas have been used (e.g., in \cite{CKRS}) to give more
refined conjectures as to the asymptotic size of $V_E(x)$.  They can
also be used to compute a large number of central critical values of
twisted $L$-functions as, generally, computing coefficients of a form
is easier than computing the central critical values of an
$L$-function. Moreover, since the Fourier coefficients are integral,
the computation allows one to determine exactly whether or not there is a vanishing as there will be no error introduced by rounding.


In the 1980s and 1990s similar formulas were conjectured for Siegel
modular forms.  In \cite{Bocherer} a conjecture was formulated
relating central critical values of quadratic twists of spinor
$L$-functions associated to a Siegel modular form $F$ of full level
and sums of coefficients of $F$.  The conjecture when $F$ is a Maass
lift was proved in \cite{Bocherer}. Numerical evidence for the
conjecture when $F$ is not a Maass lift was provided in
\cite{KohnenKuss} for rational eigenforms of level 1 and weight up to
26, and only recently for higher weights and non-rational eigenforms in
\cite{Raum}.  For modular forms of (squarefree) level $N>1$ the
conjecture has been proved in \cite{BochererSchulzePillot} in the case
when the form is a Yoshida lift.

In what follows we investigate a version of Böcherer's Conjecture in
the setting of paramodular forms.  Paramodular forms are similar to Siegel
modular forms on $\Sp{4}{\Z}$ in that they are multivariate modular
forms on a group of rank 4.  Recently they have been studied in the
context of the Paramodular Conjecture which identifies the spinor
$L$-functions of certain paramodular forms with the Hasse-Weil
$L$-functions of certain abelian surfaces \cite{PoorYuen}, \cite{BrumerKramer}.  

We give evidence for our generalization of Böcherer's Conjecture in
two ways: first, we prove the conjecture when the form is a Gritsenko
lift and, second, we verify the conjecture computationally in a number
of cases.  For a fundamental discriminant $D<0$ coprime to the level,
our conjecture takes the form
\[
L(F,1/2,\chi_D) = C_F \abs{D}^{1-k} A(D)^2
\]
where $F$ is a paramodular form, $C_F>0$ is a constant that only
depends on $F$, and $A(D)$ is an
average of the coefficients of $F$.  It turns out that computing the
Fourier coefficients of such an $F$ is computationally very expensive: we use
data from \cite{PoorYuen} to compute the right-hand side of the
formula.

For the left-hand side, however, the data of \cite{PoorYuen} yield at most the Euler factors at the primes 2, 3, 5, 7 and, in some cases, 11.  In
particular, in all cases the first $L$-series coefficient that is unknown is less than 17.  In
order to compute central critical values to any reasonable precision,
we do not have enough coefficients.
Instead, we assume the Paramodular Conjecture and compute
central critical values $L(F,1/2,\chi_D)$ for many $D$ by computing
the central critical value of the corresponding Hasse-Weil $L$-functions and showing
that, numerically, $L(F,1/2,\chi_D) / (C_F \abs{D}^{1-k})$ is the square
of $A(D)$ for the $D$ for which we have data. The constant $C_F$ is
computed from the formula applied to the smallest possible $\abs{D}$.

The paper is organized as follows.  In the rest of this section we
define all the terminology and notation needed to understand the
conjecture.  In the following section we prove the conjecture for
paramodular forms that are lifts. In the third section we describe our
experimental results.  The fourth section deals with the
case when the form $F$ is in the minus space with respect to the
Atkin-Lehner operator---the conjecture holds in the case but some
interesting computational phenomena arise.  We conclude with tables reporting the data we generated.

\subsection{Notation}  
The main objects of study in this paper are paramodular forms of prime
level $p$ and their $L$-functions.
\par
Let $R$ be a commutative ring with identity. The symplectic group
is $\Sp{4}{R}:=\{x\in \GL{4}{R}:  x' J_2 x = J_2\}$,
where the transpose of matrix $x$ is denoted $x'$ and for the $n
\times n$ identity matrix $I_n$ we set
$J_n = \left(\begin{smallmatrix}
0 & I_n\\-I_n&0
\end{smallmatrix}\right)$.
When $R\subset \R$, the group of symplectic similitudes
is $\GSp{4}{R} := \{x\in\GL{4}{R}: \exists \mu\in\R_{>0}: x' J_2 x =
\mu J_2\}$.

The paramodular group of level $p$ is
\begin{equation*}\label{eq:para}
\para{p} := \Sp{4}{\mathbb{Q}}\cap 
\begin{pmatrix} 
* &* & */p &*\\
p* & * &*& *\\
p*& p*& * & p*\\
p* & * & * & *
\end{pmatrix},
\text{ where $*\in\Z$.}
\end{equation*}

\subsection{Modular Form Notation}  Let $\H_n:=\{ Z = X+iY\in
M_{n\times n}(\C): Z' = Z, Y>0\}$ be the Siegel upper half space.  The
group $\GSp{4}{\R}$ acts on $\H_2$ by $\gamma\langle Z \rangle =
(AZ+B)(CZ+D)^{-1}$ where $\gamma=\left(\begin{smallmatrix} A&B\\C &
D\end{smallmatrix}\right)$.  The complex vector space of paramodular
forms of degree 2, level $p$ and weight $k$ is the set of holomorphic
$F:\H_2\to\C$ such that \[
(F|_k\gamma)(Z):=\det(CZ+D)^{-k}F(\gamma\langle Z\rangle) =   F(Z)
\]
for all $\gamma\in \para{p}$ and such that for all positive definite
$Y_0$ and for all $\gamma\in\Sp{4}{\Z}$, $F|\gamma$ is bounded on
$\{Z\in\H_2: \Im{Z}>Y_0\}$.  We denote the space of paramodular forms
by $M^k(\para{p})$.  For $F\in M^k(\para{p})$ we define the Siegel
$\Phi$-operator as $\Phi(F)(Z):=\lim_{\lambda\to\infty}
F\left(\left(\begin{smallmatrix} i\lambda &
0\\0&Z\end{smallmatrix}\right)\right)$
and the space of cusp forms
$S^k(\para{p})$ as the space of all paramodular forms so that $F\mid \gamma\in\ker \Phi$ for all cusps $\gamma$.

By the K\"{o}cher principle, every $F\in M^k(\para{p})$ has a Fourier expansion of the form
\[
F(Z) = \sum_{T\in {}^p\mathcal{X}_2} a(T;F) q^{pm}\zeta^rq'^{n}
\]
where $q := e^{2\pi i z}, q':=e^{2\pi i z'}~ (z,z'\in \H_1)~ \zeta := e^{2 \pi i \tau} ~(\tau\in\C)$
and
\[
{}^p\mathcal{X}_2 := \left\{\left(\begin{smallmatrix} pa&b/2\\b/2&c\end{smallmatrix}\right)
\geq 0
\;:\;
a,b,c\in\Z
\right\}.
\]
For $F\in S^k(\para{p})$, we have $a(T[U];f)=\det(U)^ka(T;f)$ for
every $U\in\hat{\Gamma}_0(p)$ where $\hat{\Gamma}_0(p) := \langle
\Gamma_0(p),\left(\begin{smallmatrix} 1&0
\\0&-1\end{smallmatrix}\right)\rangle$
(here $\Gamma_0(p)$ is the
congruence subgroup of $\SL$ with lower lefthand entry congruent to 0
mod $p$.)  Moreover, cusp forms are supported on the positive definite
matrices in ${}^p\mathcal{X}_2$.

The space $S^k(\para{p})$ can be split into a plus space and a minus space.  Define an operator
\[
\mu_p = \frac{1}{\sqrt{p}}
\begin{pmatrix}
0 & 1 & 0 & 0  \\
-p & 0 & 0 & 0 \\
0 & 0 & 0 & p \\
0 & 0 & -1 & 0
\end{pmatrix},
\]
an involution.
Then, we define $S^k(\para{p})^\pm =\{f\in S^k(\para{p}): f\mid \mu = \pm f\}$.  

\subsection{$L$-function Notation}


Following Andrianov\cite{Andrianov}, one can define operators $T(n)$
in terms of double cosets.  Specifically, using \cite{Ibukiyama}, one
can define the action of the operator $T(q^\delta)$ for
$S^k(\para{p})$ for primes $(p,q)=1$.


Suppose we are given a paramodular form $F\in S^k(\para{p})$ so that
for all $n\in\Z$, $F|T(n) = \lambda_{F,n}F=\lambda_n F$.  Then we can
define the spinor $L$-series by the Euler product
\begin{equation}\label{eq:spinor}
L(F,s) := \prod_{\text{$q$ prime}} L_q\bigl(q^{-s-k+3/2})^{-1},
\end{equation}
where the local Euler factors are given by
\[
L_q(X) := 1 - \lambda_q X + (\lambda_q^2-\lambda_{q^2}-q^{2k-4}) X^2
            - \lambda_q q^{2k-3} X^3 + q^{4k-6} X^4
\]
for $q\neq p$, and $L_p(X)$ has a similar formula but of degree 3 (this will be investigated further by the first author in future work).


As our computations are in weight 2, we will compute the $L$-series of a paramodular form of prime level by assuming the following conjecture.
\begin{conjecture}[Paramodular Conjecture]\label{conj:paramodular}
Let $p$ be a prime.  There is a bijection between lines of Hecke
eigenforms $F\in S^2(\para{p})$ that have rational eigenvalues and are
not Gritsenko lifts and isogeny classes of rational abelian surfaces
$\mathcal{A}$ of conductor $p$.  In this correspondence we have that
\[
L(\mathcal{A},s,\text{Hasse-Weil})=L(F,s).
\]
\end{conjecture}
We remark that it is a conjecture that the two $L$-series mentioned above have an analytic continuation and satisfy a functional equation.  


In order to compute central values we need the Selberg data for the $L$-function:  specifically for an $L$-series $L(s) = \sum_{n\geq 0} a(n)n^{-s}$ 
\begin{itemize}
\item we complete $L(s)$ by multiplying it by some $\Gamma$-factors of the form
\[
\gamma(s) := \Gamma\left(\frac{s+\lambda_1}{2}\right)\cdots \Gamma\left(\frac{s+\lambda_d}{2}\right)
\]
and an exponential factor $A^s$; i.e., $L^*(s) :=A^s\gamma(s)L(s)$ and
we need $\lambda_1,\dots,\lambda_d$ and $A$; and
\item we require that $L^*(s)$ satisfies a functional equation of the
form
\[
L^*(s) = \pm L^*(1-s).
\]
\end{itemize}
We note that we use the analytic normalization $s\mapsto 1-s$ and that the factor $\gamma(s)$ is not unique as it can be rewritten using the duplication formula.

A table in \cite{Dokchitser} summarizes the data that we use:  in
particular, as we have a degree 4 $L$-function attached to a
paramodular form of weight 2 and level $p$ (that corresponds to an abelian surface isogenous to a 
curve of genus 2) we have 
\begin{equation}\label{eq:functional}
L^*(F,s) = \left(\frac{\sqrt{p}}{4\pi^2}\right)^s\Gamma(s+1/2)\Gamma(s+1/2)L(F,s).
\end{equation}
so that conjecturally
\[
L^*(F,s) = \epsilon\, L^*(F,1-s),
\]
when $F\in S^2(\para{p})^\epsilon$.


Let $D$ be a fundamental discriminant,
and denote by $\chi_D$ the
unique quadratic character of conductor $D$.  For the spinor
$L$-series $L(F,s) = \sum_{n\geq 1} a(n)\,n^{-s}$ of a paramodular form
$F$, we define the quadratic twist
\[
L(F,s,\chi_D) := \sum_{n \geq 1} \chi_D(n) a(n)\,n^{-s}.
\]
In our case, most of the Selberg data for the $L$-function is expected
to be the same as the data for the non-twisted $L$-function, except
for the exponential factor and the sign of the functional equation.
For instance, assuming $p\nmid D$,
the exponential factor increases by a factor of $D^2$
and the sign of
the functional equation is changed by a factor of $\kro{D}{p}$.

The computation of the central values of the $L$-functions was done
using Mike Rubinstein's \textsf{lcalc} package \cite{lcalc}.  

\subsection{Gritsenko Lifts}\label{sec:gritsenko}

A Gritsenko lift \cite{Gritsenko} is a paramodular form that comes from a Jacobi form.  The standard reference for Jacobi forms is \cite{EichlerZagier}.  We summarize the relevant terminology here.  
\begin{definition}
A Jacobi form of level 1, weight $k$ and index $m$ is a function $\phi(z,\tau)$ for $z\in \H_1$ and $\tau\in \C$ such that:
\begin{enumerate}
\item $\phi\left(\frac{az+b}{cz+d},\frac{\tau}{cz+d}\right)=(cz+d)^ke^{\frac{2\pi i m c \tau^2}{cz+d}}\phi(\tau,z)$ for $\left(\begin{smallmatrix} a& b\\c&d\end{smallmatrix}\right)\in\SL$;
\item $\phi(z,\tau+ \lambda z+\mu)=e^{-2\pi i m(\lambda^2z+2\lambda \tau)}\phi(z,\tau)$ for all integers $\lambda,\mu$; and
\item $\phi$ has a Fourier expansion 
\[
\phi(z,\tau)=\sum_{n\geq 0}\sum_{r^2\leq 4nm} c(n,r)q^n\zeta^r.
\]
\end{enumerate}
\end{definition}

Our first main theorem is the proof of Conjecture \ref{conj:bocherer} for Gritsenko lifts and so we make the following definition:

\begin{definition}[Gritsenko Lift]  \label{def:gritsenko}
Let $\phi\in S_{k,p}$ and suppose
$\phi(\tau,z)=\sum_{n>0,r\in\Z}c(n,r)q^n\zeta^r$ is its Fourier
expansion.  Then the Gritsenko lift of $\phi$ is $\text{Grit}(\phi)\in
S^k(\para{p})^+$ given by
\begin{equation}\label{eq:grit}
\text{Grit}(\phi)\left(\begin{smallmatrix} \tau & z\\z& w\end{smallmatrix}\right) = \sum_{n,r,m}\left(\sum_{\delta|(n,r,m)}\delta^{k-1}c\left(\frac{mn}{\delta^2},\frac{r}{\delta}\right)\right)q^{mp}\zeta^r q'^n.
\end{equation}
\end{definition}

We remark that it also makes sense to talk about $\text{Grit}(f)$
where $f$ is a modular cuspform of level $p$, in the minus space and weight $2k-2$ that
corresponds to a $\phi\in S_{k,p}$ (this can be done via the inverse of the map constructed in Theorem~5 of \cite{SkoruppaZagier}).

\subsection{Summary of Main Results}

We define 
\[
A_F(D):=\sum_{\{T>0\;:\;\disc T=D\}/\hat{\Gamma}_0(p)}\frac{a(T;F)}{\varepsilon(T)}
\]
where $\varepsilon(T):=\# \{U\in\hat{\Gamma}_0(p):T[U]=T\}$.  We often write $A(D)$ when $F$ is obvious from context.  Our main goal is to give evidence for the following conjecture

\begin{conjecture}[Paramodular B\"{o}cherer's Conjecture]\label{conj:bocherer}
Suppose $F\in S^k(\para{p})^+$.  Then, for fundamental discriminants
$D<0$ we have
\begin{equation*}
L(F,1/2,\chi_D) = \star\, C_F \abs{D}^{1-k} A(D)^2
\end{equation*}
where $C_F$ is a positive constant that depends only on $F$,
and $\star=1$ when $p\nmid D$, and $\star=2$ when $p\mid D$.
\end{conjecture}

The evidence we provide is in two forms.  First, we prove the
conjecture in the case that the form is a Gritsenko lift:

\begin{theorem}\label{thm:grit}
Let $F=\text{Grit}(f)\in S^k(\para{p})^+$ where $p$ is prime and $f$
is a Hecke eigenform of degree 1, level $p$ and weight $2k-2$.  Then there exists
a constant $C_F>0$ so that
\[
L(F,1/2,\chi_D) = \star\, C_F \abs{D}^{1-k} A(D)^2
\]
for $D<0$ a fundamental discriminant,
and $\star=1$ when $p\nmid D$, and $\star=2$ when $p\mid D$.
\end{theorem}

The idea of the proof is to combine four ingredients: (i) the
factorization of the $L$-function of the Gritsenko lift as in
\cite{SchmidtPara}, (ii) Dirichlet's class number formula, (iii) an
explicit description of the Fourier coefficients of the Gritsenko lift
and (iv) Waldspurger's theorem relating the central values of
quadratic twists to sums of coefficients of modular forms of
half-integer weight \cite{Waldspurger}.

Second, we verify the conjecture computationally in a number of cases.
The computations we do are based on \cite{PoorYuen} and \cite{Stoll}.  On the one hand, \cite{PoorYuen} provides Fourier coefficients for all paramodular forms of prime level up to 600 that are not Gritsenko lifts and on the other \cite{Stoll} provides most 
of the curves that correspond to the paramodular forms.  Armand Brumer kindly provided a curve that was not in \cite{Stoll} but corresponds to one of the paramodular forms in \cite{PoorYuen}.  By matching levels of modular forms and discriminants of hyperelliptic curves we show the following complement to Theorem \ref{thm:grit}.  Suppose $F\in S^2(\para{p})^+$ for $p$ a prime less than 600 is not a
Gritsenko lift.  Then, numerically, there exists a coefficient $C_F>0$
so that 
\[
L(F,1/2,\chi_D) = \star\, C_F \abs{D}^{1-k} A(D)^2
\]
for $D<0$ a fundamental discriminant listed in Tables~\ref{tbl:values277}--\ref{tbl:values587p}.

Note that in case $\kro{D}{p} = -1$ the twisted central value is
expected to be zero due to the sign of the functional equation being
$-1$, and on the right hand side the average $A(D)$ is an empty sum.
For this reason, we exclude these discriminants from our computation.

\section{The Case of Lifts}\label{sec:lifts}


Assume $p$ is an odd prime. Recall (\cite[Theorem 2.2, p.
23]{EichlerZagier}) that $c(n, r)$ depends only on $D=r^2 - 4np$; call
this number $c^\ast(D)$, i.e.
\[
  c^\ast(D) := c\left(\frac{r^2-D}{4p}, r \right),
\]
for any $r\in\Z$ such that $r^2\equiv D\pmod{4p}$. We let
$c^\ast(D):=0$ otherwise.

\begin{lemma}
Let $D$ be a fundamental discriminant. Then
\[
\sum_{\substack{T\in{}^p\mathcal{X}_2/\hat\Gamma_0(p)\\ \disc T = D}}
\frac{1}{\varepsilon(T)}
=
\frac{h(D)}{w_D},
\]
where $h(D)$ and $w_D$ are the class number and the number of units of
the quadratic order of discriminant $D$, respectively.
\end{lemma}
\begin{proof}
We start by noting that 
$\# \{U\in{\Gamma}_0(p):T[U]=T\} = w_D$
for any $T$ with $\disc T = D$.
Also, 
$\#\, {{}^p\mathcal{X}_2/\Gamma_0(p)} = 2\,h(D)$. 

The lemma then follows from the fact that
\[
\sum_{\substack{T\in{}^p\mathcal{X}_2/\hat\Gamma_0(p)\\ \disc T = D}}
\frac{1}{\varepsilon(T)}
= \frac{1}{[\hat\Gamma_0(p):\Gamma_0(p)]} \sum_{\substack{T\in{}^p\mathcal{X}_2/\Gamma_0(p)\\ \disc T = D}}
\frac{1}{w_D}
= \frac{1}{2w_D} 2h(D)
\]
\end{proof}

\begin{proposition}
\label{prop1}
Let $F=\text{Grit}(\phi)$. For $D<0$ a fundamental discriminant we
have:
\[
   A(D ; F) = \frac{h(D)}{w_D} \, c^\ast(D)
\]
\end{proposition}
\begin{proof}
By the definition of the Gritsenko lift, we know that
\[
   A(T ; F) = c^\ast(\disc T)
\]
provided $T$ is primitive; this is always the case when $\disc T=D$ is a
fundamental discriminant.

Thus
\[
A(D ; F)
=\sum_{\substack{T\in{}^p\mathcal{X}_2/\hat\Gamma_0(p)\\ \disc T = D}}
\frac{c^\ast(D)}{\varepsilon(T)}
\]
The result follows from the lemma.
\end{proof}

\begin{proposition}
\label{prop2}
Let $F = Grit(f) \in S^k(\para{p})^+$.
\begin{enumerate}
\item $L(F, s, \chi_D)$ has an analytic continuation to an entire
function.
\item $L(F, 1/2, \chi_D) = \frac{4\pi^2}{w_D^2}\cdot
\frac{h(D)^2}{\sqrt{\abs{D}}} \cdot L(f, 1/2, \chi_D)$
where $D<0$ is a fundamental discriminant.
\end{enumerate}
\end{proposition}
\begin{proof}
It is a standard fact that $L(F,s) = \zeta(s+1/2)\, \zeta(s-1/2)\,
L(f,s)$ (using the analytic normalization, so that the center is at
$s=1/2$). Twisting by $\chi_D$ we obtain
\[
   L(F,s,\chi_D) = L(s+1/2,\chi_D)\, L(s-1/2,\chi_D)\, L(f,s,\chi_D)
\]
valid on the region of convergence. Since the right hand side has an
analytic continuation, (1) follows.
\par
To prove (2), we evaluate the above equation at $s=1/2$, and use
the Dirichlet class number formula for $L(0,\chi_D)$ and
$L(1,\chi_D)$.
\end{proof}

\begin{proof}[Proof of Theorem~\ref{thm:grit}]
By Waldspurger's formula \cite{Waldspurger}, we have
\[
   L(f, 1/2, \chi_D) = \star\, k_f \cdot \frac{c^\ast(D)^2}{\abs{D}^{k-3/2}},
\]
with $k_f > 0$.
The theorem thus follows directly from Proposition~\ref{prop1} and part (2) of
Proposition~\ref{prop2}.
\end{proof}

\section{The Case of Nonlifts}\label{sec:nonlifts}

In this section we describe numerical experiments that support Conjecture \ref{conj:bocherer} in the case when the form is not a Gritsenko lift.  We discuss how to compute the $L$-series that correspond to paramodular forms of weight 2 and prime conductor $p<600$.  For the rest of this section, let $F$ be such a form.  In order to compute the central values
$L(F,1/2,\chi_D)$ for several twists $\chi_D$ we would need a large
number of coefficients of $F$ since the exponential factor grows
as $D^2$. As already mentioned in the introduction, the data
in \cite{PoorYuen} are not enough for this purpose.

To remedy this we do the following.  In the same paper \cite{PoorYuen}, Poor and Yuen describe and verify Conjecture \ref{conj:paramodular}.  For our purposes, this conjecture asserts that to compute the $L$-series of $F$, we can compute the Hasse-Weil $L$-series of a related Abelian surface $\mathcal{A}$.  In particular, Table \ref{tbl:curves} associates each $F$ to a hyperelliptic curve $C$ isogenous to $\mathcal{A}$.  There are two forms of level 587 that are nonlifts; one is in the plus space the other is in the minus space (see Section \ref{sec:minus}).
\begin{table}
\begin{center}
\caption{Hyperelliptic curves $C$ used to compute $L$-series associated to paramodular forms of level $p$ that are not lifts.}\label{tbl:curves}
\def\eq#1{\hspace{19em}\llap{#1}\hspace{1em}}%
\begin{tabular}{|c|c|c|c|}\hline
$p$ & $\epsilon$ & $\lambda$ & $C$ \\\hline\hline
277 & + &   8 &\eq{$y^{2} + y = x^{5} - 2x^{3} + 2x^{2} - x$}                \\\hline
349 & + &  12 &\eq{$y^{2} + y = -x^{5} - 2x^{4} - x^{3} + x^{2} + x$}        \\\hline
353 & + &  -9 &\eq{$y^{2} + \left(x^{3} + x + 1\right) y = x^{2}$}           \\\hline
389 & + & -10 &\eq{$y^{2} + x y = -x^{5} - 3x^{4} - 4x^{3} - 3x^{2} - x$}    \\\hline
461 & + &   0 &\eq{$y^{2} + y = -2x^{6} + 3x^{5} - 3x^{3} + x$}              \\\hline
523 & + &  24 &\eq{$y^{2} + x y = -x^{5} + 4x^{4} - 5x^{3} + x^{2} + x$}     \\\hline
587 & + &  -6 &\eq{$y^{2} = -3x^{6} + 18x^{4} + 6x^{3} + 9x^{2} - 54x + 57$} \\\hline
587 & - & -36 &\eq{$y^{2} + \left(x^{3} + x + 1\right) y = -x^{3} - x^{2}$}  \\\hline
\end{tabular}
\end{center}
\end{table}

By the Paramodular Conjecture, then, the $L$-function of the curve $C$ of conductor $p$ is equal to the $L$-function corresponding to the paramodular form $F$ of level $p$.  Thus, we compute the $L$-function by counting points on the curve.  In computing the $L$-function of $F$ in this way, we provide evidence for the Paramodular Conjecture as well.  Since the Paramodular B\"ocherer's Conjecture is verified for the $L$-function of $F$ computed via this correspondence, it strongly suggests that the Hasse-Weil and spinor $L$-functions agree.

Assuming Conjecture \ref{conj:paramodular}, we compute the $L$-series
for $F$ by counting points on its corresponding $C$.  The reciprocal
of the $q$-th Euler factor of $L$-series in \eqref{eq:spinor}
for $q\neq p$,
specialized to $k=2$ is of the form
\[
L_q(X) =
1-\lambda_q\, X  +  (\lambda_q^2-\lambda_{q^2}-1)\, X^2
  -\lambda_q\,q\, X^3  +  q^2\, X^4.
\]
Writing the Euler factor as $L_q(X)^{-1} = \sum_{i\geq 0} a(q^i)\,X^i$ and matching
it up with the Hasse-Weil $L$-series allows us to conclude
\begin{align*}
\lambda_q &= a(q) = 1+q - N_1 \\
\lambda_{q^2} + 1 &= a(q^2) = 1 + q + q^2 - (1+q)\,N_1) + (N_1^2-N_2)/2
\end{align*}
where $N_1$ is the number of points on $C$ over $\mathbb{F}_q$ and
$N_2$ is the number of points on $C$ over $\mathbb{F}_{q^2}$.
We determined $N_1$ and $N_2$ by counting points on $C$ 
using Sage \cite{sage}.
The $p$-th Euler factor of $C$ is of degree $3$ and given by
\[
  L_p(X) = (1 - \epsilon X)\,(1 - \lambda\, X + p X^2),
\]
where $\epsilon$ and $\lambda$ for each $C$ are given in Table~\ref{tbl:curves}.

Having all the local Euler factors, we computed
the central value of the $L$-function and its quadratic twists using
Mike Rubinstein's \textsf{lcalc} \cite{lcalc}.

We recall Conjecture \ref{conj:bocherer}
\[
\frac{L(F,s,\chi_D)}{C_F}\abs{D} = \star\,\left(\sum_{\{T>0\;:\;\disc T=D\}/\hat{\Gamma}_0(p)}\frac{a(T;F)}{\varepsilon(T)}\right)^2
\]
where the constant at $C_F$ is positive.  In Tables
\ref{tbl:values277}--\ref{tbl:values587p} the Conjecture is verified in the case of forms that are not Gritsenko lifts and have been computed in \cite{PoorYuen}.  We first determine $C_F$ by solving for
it in the case of the first discriminant $D$ in the table.  The second
column is the average $A(D)$ of the coefficients of discriminant $D$
of the form $F$.  The third column is the quantity $
\frac{L({F},1/2,\chi_D)}{C_F}\,\lvert{D}\rvert$ which,
numerically, for the discriminants $\abs{D}<200$ are $A(D)^2$.
For the paramodular form of level $277$, we also include data for
$D=-3\cdot277$ and $D=-4\cdot277$; in this case, since $p\mid D$, the
quantity in the third column is expected to agree with $2A(D)^2$.

All the computations described above were done using an eight core Xeon E5520
system. Computing the first $10^6$ coefficients of the Hasse-Weil
$L$-series
for the eight curves took a total of about 60 cpu-days using a combination
of Sage and custom code written in Python and Cython. Computing the
central values of the $L$-functions and their quadratic twists for the
discriminants with $\abs{D}<200$ took less than 1 cpu-hour using \textsf{lcalc}.

\section{The minus space}\label{sec:minus}

Suppose $F\in S^k(\para{p})^-$, and let $D<0$ be a fundamental
discriminant.
In case $\kro{D}{p}=+1$,
the formula of Conjecture~\ref{conj:bocherer} holds trivially.
Indeed, note that for such $F$ the
sign of the functional equation is $-1$ and so the central critical
value $L(F,s,\chi_D)$ is zero.  On the other hand, $A(D)$ can be show
to be zero in the following way:  Poor and Yuen \cite[Definition
3.9]{PoorYuen} define an involution $\text{Twin}$ over the set
$\mathcal{X}_2^p$ with discriminant $D$ for which
$a(\text{Twin}(T);F)=a(T;F\mid\mu)$.  Since we are in the minus space
\begin{align*}
A(D) &=\sum_{\{T>0\;:\;\disc T=D\}/\hat{\Gamma}_0(p)}\frac{a(T;F)}{\varepsilon(T)}\\
&=\sum_{\{T>0\;:;\disc T=D\}/\hat{\Gamma}_0(p)}\frac{a(\text{Twin}(T);F)}{\varepsilon(T)}\\
&=\sum_{\{T>0\;:\;\disc T=D\}/\hat{\Gamma}_0(p)}\frac{a(T;F\mid\mu)}{\varepsilon(T)}\\
&=-\sum_{\{T>0\;:\;\disc T=D\}/\hat{\Gamma}_0(p)}\frac{a(T;F)}{\varepsilon(T)}\\
&=-A(D).
\end{align*}

On the other hand, the formula of Conjecture~\ref{conj:bocherer} fails
to hold in case $\kro{D}{p}=-1$. Since $A(D)$ is an empty sum for
this type of discriminants, the right hand side of the formula
vanishes trivially. However, the left hand side is still an
interesting central value, not necessarily vanishing.
For example, as can be seen in Table \ref{tbl:values587m} we
have for the nonlift in $S^2(\para{587})^-$ that
$\frac{L({F_{587}^-},1/2,\chi_D)}{{C_{{587}}^-}}\,\abs{D}$
seems to always be the square of an integer, and frequently nonzero,
in spite of $A(D)$ being zero.
In a future paper we will investigate this phenomenon.

\begin{table}[p]
\caption{Data for the paramodular form of level 277, based on the Hasse-Weil $L$-series for the curve $y^{2} + y = x^{5} - 2x^{3} + 2x^{2} - x$. The constant ${C_{277}} = 6.49630674438$.}
\label{tbl:values277}
\begin{tabular}{|c|c|c||c|c|c|}
\hline & & & & & \\[-2ex]
$D$ & $A(D; {F_{277}})$ & $\frac{L({F_{277}},1/2,\chi_D)}{{C_{277}}}\,\abs{D}$
& $D$ & $A(D; {F_{277}})$ & $\frac{L({F_{277}},1/2,\chi_D)}{{C_{277}}}\,\abs{D}$
\\[0.8ex] \hline\hline
\phantom{00}-3\phantom{-0} & \phantom{0}-1\phantom{-0} & \phantom{-00}1.000000\phantom{-0}        &        
\phantom{0}-83\phantom{-0} & \phantom{-0}6\phantom{-0} & \phantom{-0}36.000000\phantom{-0} \\\hline        
\phantom{00}-4\phantom{-0} & \phantom{0}-1\phantom{-0} & \phantom{-00}1.000000\phantom{-0}        &        
\phantom{0}-84\phantom{-0} & \phantom{-0}1\phantom{-0} & \phantom{-00}1.000000\phantom{-0} \\\hline        
\phantom{00}-7\phantom{-0} & \phantom{0}-1\phantom{-0} & \phantom{-00}1.000000\phantom{-0}        &        
\phantom{0}-87\phantom{-0} & \phantom{0}-3\phantom{-0} & \phantom{-00}9.000000\phantom{-0} \\\hline        
\phantom{0}-19\phantom{-0} & \phantom{0}-2\phantom{-0} & \phantom{-00}4.000000\phantom{-0}        &        
\phantom{0}-88\phantom{-0} & \phantom{0}-2\phantom{-0} & \phantom{-00}4.000000\phantom{-0} \\\hline        
\phantom{0}-23\phantom{-0} & \phantom{00}0\phantom{-0} & \phantom{00}-0.000000\phantom{-0}        &        
\phantom{0}-91\phantom{-0} & \phantom{0}-1\phantom{-0} & \phantom{-00}1.000000\phantom{-0} \\\hline        
\phantom{0}-39\phantom{-0} & \phantom{-0}1\phantom{-0} & \phantom{-00}1.000000\phantom{-0}        &        
\phantom{}-116\phantom{-0} & \phantom{-0}3\phantom{-0} & \phantom{-00}9.000000\phantom{-0} \\\hline        
\phantom{0}-40\phantom{-0} & \phantom{0}-6\phantom{-0} & \phantom{-0}36.000000\phantom{-0}        &        
\phantom{}-120\phantom{-0} & \phantom{0}-2\phantom{-0} & \phantom{-00}4.000000\phantom{-0} \\\hline        
\phantom{0}-47\phantom{-0} & \phantom{00}0\phantom{-0} & \phantom{-00}0.000000\phantom{-0}        &        
\phantom{}-123\phantom{-0} & \phantom{0}-1\phantom{-0} & \phantom{-00}1.000000\phantom{-0} \\\hline        
\phantom{0}-52\phantom{-0} & \phantom{-0}5\phantom{-0} & \phantom{-0}25.000000\phantom{-0}        &        
\phantom{}-131\phantom{-0} & \phantom{}-10\phantom{-0} & \phantom{-}100.000000\phantom{-0} \\\hline        
\phantom{0}-55\phantom{-0} & \phantom{0}-2\phantom{-0} & \phantom{-00}4.000000\phantom{-0}        &        
\phantom{}-136\phantom{-0} & \phantom{0}-6\phantom{-0} & \phantom{-0}36.000000\phantom{-0} \\\hline        
\phantom{0}-59\phantom{-0} & \phantom{-0}3\phantom{-0} & \phantom{-00}9.000000\phantom{-0}        &        
\phantom{}-155\phantom{-0} & \phantom{}-10\phantom{-0} & \phantom{-}100.000000\phantom{-0} \\\hline        
\phantom{0}-67\phantom{-0} & \phantom{0}-8\phantom{-0} & \phantom{-0}64.000000\phantom{-0}        &        
\phantom{}-164\phantom{-0} & \phantom{0}-5\phantom{-0} & \phantom{-0}25.000000\phantom{-0} \\\hline        
\phantom{0}-71\phantom{-0} & \phantom{-0}2\phantom{-0} & \phantom{-00}4.000000\phantom{-0}        &        
\phantom{}-187\phantom{-0} & \phantom{-0}8\phantom{-0} & \phantom{-0}64.000001\phantom{-0} \\\hline        
\phantom{0}-79\phantom{-0} & \phantom{00}0\phantom{-0} & \phantom{-00}0.000000\phantom{-0}        &        
\phantom{}-191\phantom{-0} & \phantom{-0}2\phantom{-0} & \phantom{-00}3.999999\phantom{-0} \\\hline        
\end{tabular}
\end{table}

\begin{table}[p]
\caption{Data for the paramodular form of level 349, based on the Hasse-Weil $L$-series for the curve $y^{2} + y = -x^{5} - 2x^{4} - x^{3} + x^{2} + x$. The constant ${C_{349}} = 7.91921340249$.}
\label{tbl:values349}
\begin{tabular}{|c|c|c||c|c|c|}
\hline & & & & & \\[-2ex]
$D$ & $A(D; {F_{349}})$ & $\frac{L({F_{349}},1/2,\chi_D)}{{C_{349}}}\,\abs{D}$
& $D$ & $A(D; {F_{349}})$ & $\frac{L({F_{349}},1/2,\chi_D)}{{C_{349}}}\,\abs{D}$
\\[0.8ex] \hline\hline
\phantom{00}-3\phantom{-0} & \phantom{-0}1\phantom{-0} & \phantom{-00}1.000000\phantom{-0}        &        
\phantom{0}-95\phantom{-0} & \phantom{0}-4\phantom{-0} & \phantom{-0}15.999431\phantom{-0} \\\hline        
\phantom{00}-4\phantom{-0} & \phantom{-0}1\phantom{-0} & \phantom{-00}1.000000\phantom{-0}        &        
\phantom{}-104\phantom{-0} & \phantom{-0}2\phantom{-0} & \phantom{-00}4.001714\phantom{-0} \\\hline        
\phantom{0}-15\phantom{-0} & \phantom{0}-1\phantom{-0} & \phantom{-00}1.000000\phantom{-0}        &        
\phantom{}-111\phantom{-0} & \phantom{0}-1\phantom{-0} & \phantom{-00}0.986454\phantom{-0} \\\hline        
\phantom{0}-19\phantom{-0} & \phantom{-0}4\phantom{-0} & \phantom{-0}16.000000\phantom{-0}        &        
\phantom{}-115\phantom{-0} & \phantom{-}14\phantom{-0} & \phantom{-}196.009480\phantom{-0} \\\hline        
\phantom{0}-20\phantom{-0} & \phantom{-0}1\phantom{-0} & \phantom{-00}1.000000\phantom{-0}        &        
\phantom{}-116\phantom{-0} & \phantom{-0}6\phantom{-0} & \phantom{-0}36.007307\phantom{-0} \\\hline        
\phantom{0}-23\phantom{-0} & \phantom{00}0\phantom{-0} & \phantom{00}-0.000000\phantom{-0}        &        
\phantom{}-123\phantom{-0} & \phantom{-}18\phantom{-0} & \phantom{-}323.991266\phantom{-0} \\\hline        
\phantom{0}-31\phantom{-0} & \phantom{0}-1\phantom{-0} & \phantom{-00}1.000000\phantom{-0}        &        
\phantom{}-139\phantom{-0} & \phantom{}-11\phantom{-0} & \phantom{-}120.948205\phantom{-0} \\\hline        
\phantom{0}-51\phantom{-0} & \phantom{-0}5\phantom{-0} & \phantom{-0}25.000000\phantom{-0}        &        
\phantom{}-143\phantom{-0} & \phantom{0}-2\phantom{-0} & \phantom{-00}4.045737\phantom{-0} \\\hline        
\phantom{0}-56\phantom{-0} & \phantom{-0}2\phantom{-0} & \phantom{-00}4.000001\phantom{-0}        &        
\phantom{}-148\phantom{-0} & \phantom{0}-1\phantom{-0} & \phantom{-00}0.949659\phantom{-0} \\\hline        
\phantom{0}-67\phantom{-0} & \phantom{-}13\phantom{-0} & \phantom{-}168.999985\phantom{-0}        &        
\phantom{}-151\phantom{-0} & \phantom{0}-1\phantom{-0} & \phantom{-00}0.948102\phantom{-0} \\\hline        
\phantom{0}-68\phantom{-0} & \phantom{-0}3\phantom{-0} & \phantom{-00}8.999976\phantom{-0}        &        
\phantom{}-155\phantom{-0} & \phantom{0}-9\phantom{-0} & \phantom{-0}81.114938\phantom{-0} \\\hline        
\phantom{0}-83\phantom{-0} & \phantom{-0}2\phantom{-0} & \phantom{-00}3.999214\phantom{-0}        &        
\phantom{}-164\phantom{-0} & \phantom{00}0\phantom{-0} & \phantom{-00}0.144191\phantom{-0} \\\hline        
\phantom{0}-87\phantom{-0} & \phantom{0}-4\phantom{-0} & \phantom{-0}16.000248\phantom{-0}        &        
\phantom{}-168\phantom{-0} & \phantom{-0}6\phantom{-0} & \phantom{-0}36.150448\phantom{-0} \\\hline        
\phantom{0}-88\phantom{-0} & \phantom{-0}2\phantom{-0} & \phantom{-00}4.000085\phantom{-0}        &        
\phantom{}-191\phantom{-0} & \phantom{00}0\phantom{-0} & \phantom{-00}0.177733\phantom{-0} \\\hline        
\phantom{0}-91\phantom{-0} & \phantom{-0}4\phantom{-0} & \phantom{-0}15.999774\phantom{-0}        &        
& &    \\\hline
\end{tabular}
\end{table}

\begin{table}[p]
\caption{Data for the paramodular form of level 353, based on the Hasse-Weil $L$-series for the curve $y^{2} + \left(x^{3} + x + 1\right) y = x^{2}$. The constant ${C_{353}} = 9.48552733703$.}
\label{tbl:values353}
\begin{tabular}{|c|c|c||c|c|c|}
\hline & & & & & \\[-2ex]
$D$ & $A(D; {F_{353}})$ & $\frac{L({F_{353}},1/2,\chi_D)}{{C_{353}}}\,\abs{D}$
& $D$ & $A(D; {F_{353}})$ & $\frac{L({F_{353}},1/2,\chi_D)}{{C_{353}}}\,\abs{D}$
\\[0.8ex] \hline\hline
\phantom{00}-4\phantom{-0} & \phantom{-0}1\phantom{-0} & \phantom{-00}1.000000\phantom{-0}        &        
\phantom{}-111\phantom{-0} & \phantom{0}-6\phantom{-0} & \phantom{-0}36.005797\phantom{-0} \\\hline        
\phantom{00}-8\phantom{-0} & \phantom{-0}1\phantom{-0} & \phantom{-00}1.000000\phantom{-0}        &        
\phantom{}-116\phantom{-0} & \phantom{-0}2\phantom{-0} & \phantom{-00}3.989115\phantom{-0} \\\hline        
\phantom{0}-11\phantom{-0} & \phantom{-0}1\phantom{-0} & \phantom{-00}1.000000\phantom{-0}        &        
\phantom{}-120\phantom{-0} & \phantom{-0}8\phantom{-0} & \phantom{-0}63.996789\phantom{-0} \\\hline        
\phantom{0}-15\phantom{-0} & \phantom{00}0\phantom{-0} & \phantom{-00}0.000000\phantom{-0}        &        
\phantom{}-127\phantom{-0} & \phantom{-0}3\phantom{-0} & \phantom{-00}9.018957\phantom{-0} \\\hline        
\phantom{0}-19\phantom{-0} & \phantom{0}-3\phantom{-0} & \phantom{-00}9.000000\phantom{-0}        &        
\phantom{}-131\phantom{-0} & \phantom{-0}5\phantom{-0} & \phantom{-0}24.986828\phantom{-0} \\\hline        
\phantom{0}-23\phantom{-0} & \phantom{-0}1\phantom{-0} & \phantom{-00}1.000000\phantom{-0}        &        
\phantom{}-136\phantom{-0} & \phantom{0}-3\phantom{-0} & \phantom{-00}9.020983\phantom{-0} \\\hline        
\phantom{0}-35\phantom{-0} & \phantom{-0}2\phantom{-0} & \phantom{-00}4.000000\phantom{-0}        &        
\phantom{}-152\phantom{-0} & \phantom{-0}3\phantom{-0} & \phantom{-00}9.036669\phantom{-0} \\\hline        
\phantom{0}-39\phantom{-0} & \phantom{-0}2\phantom{-0} & \phantom{-00}4.000000\phantom{-0}        &        
\phantom{}-155\phantom{-0} & \phantom{-0}2\phantom{-0} & \phantom{-00}3.982909\phantom{-0} \\\hline        
\phantom{0}-43\phantom{-0} & \phantom{0}-5\phantom{-0} & \phantom{-0}25.000000\phantom{-0}        &        
\phantom{}-159\phantom{-0} & \phantom{0}-4\phantom{-0} & \phantom{-0}16.059848\phantom{-0} \\\hline        
\phantom{0}-47\phantom{-0} & \phantom{-0}1\phantom{-0} & \phantom{-00}1.000000\phantom{-0}        &        
\phantom{}-164\phantom{-0} & \phantom{-0}2\phantom{-0} & \phantom{-00}3.986694\phantom{-0} \\\hline        
\phantom{0}-68\phantom{-0} & \phantom{0}-1\phantom{-0} & \phantom{-00}1.000011\phantom{-0}        &        
\phantom{}-167\phantom{-0} & \phantom{00}0\phantom{-0} & \phantom{-00}0.018414\phantom{-0} \\\hline        
\phantom{0}-83\phantom{-0} & \phantom{0}-3\phantom{-0} & \phantom{-00}8.999872\phantom{-0}        &        
\phantom{}-168\phantom{-0} & \phantom{0}-2\phantom{-0} & \phantom{-00}4.150487\phantom{-0} \\\hline        
\phantom{0}-84\phantom{-0} & \phantom{0}-6\phantom{-0} & \phantom{-0}36.000088\phantom{-0}        &        
\phantom{}-184\phantom{-0} & \phantom{-0}9\phantom{-0} & \phantom{-0}81.067576\phantom{-0} \\\hline        
\phantom{0}-88\phantom{-0} & \phantom{-0}1\phantom{-0} & \phantom{-00}1.000097\phantom{-0}        &        
\phantom{}-187\phantom{-0} & \phantom{0}-1\phantom{-0} & \phantom{-00}0.910705\phantom{-0} \\\hline        
\phantom{0}-91\phantom{-0} & \phantom{00}0\phantom{-0} & \phantom{-00}0.000490\phantom{-0}        &        
\phantom{}-191\phantom{-0} & \phantom{-0}2\phantom{-0} & \phantom{-00}3.754734\phantom{-0} \\\hline        
\end{tabular}
\end{table}

\begin{table}[p]
\caption{Data for the paramodular form of level 389, based on the Hasse-Weil $L$-series for the curve $y^{2} + x y = -x^{5} - 3x^{4} - 4x^{3} - 3x^{2} - x$. The constant ${C_{389}} = 10.7918126629$.}
\label{tbl:values389}
\begin{tabular}{|c|c|c||c|c|c|}
\hline & & & & & \\[-2ex]
$D$ & $A(D; {F_{389}})$ & $\frac{L({F_{389}},1/2,\chi_D)}{{C_{389}}}\,\abs{D}$
& $D$ & $A(D; {F_{389}})$ & $\frac{L({F_{389}},1/2,\chi_D)}{{C_{389}}}\,\abs{D}$
\\[0.8ex] \hline\hline
\phantom{00}-4\phantom{-0} & \phantom{0}-1\phantom{-0} & \phantom{-00}1.000000\phantom{-0}        &        
\phantom{0}-91\phantom{-0} & \phantom{0}-2\phantom{-0} & \phantom{-00}3.999647\phantom{-0} \\\hline        
\phantom{00}-7\phantom{-0} & \phantom{-0}1\phantom{-0} & \phantom{-00}1.000000\phantom{-0}        &        
\phantom{0}-95\phantom{-0} & \phantom{-0}2\phantom{-0} & \phantom{-00}4.000340\phantom{-0} \\\hline        
\phantom{0}-11\phantom{-0} & \phantom{0}-1\phantom{-0} & \phantom{-00}1.000000\phantom{-0}        &        
\phantom{}-111\phantom{-0} & \phantom{00}0\phantom{-0} & \phantom{-00}0.006107\phantom{-0} \\\hline        
\phantom{0}-19\phantom{-0} & \phantom{0}-3\phantom{-0} & \phantom{-00}9.000000\phantom{-0}        &        
\phantom{}-119\phantom{-0} & \phantom{00}0\phantom{-0} & \phantom{00}-0.010524\phantom{-0} \\\hline        
\phantom{0}-20\phantom{-0} & \phantom{00}0\phantom{-0} & \phantom{-00}0.000000\phantom{-0}        &        
\phantom{}-120\phantom{-0} & \phantom{0}-8\phantom{-0} & \phantom{-0}63.995136\phantom{-0} \\\hline        
\phantom{0}-24\phantom{-0} & \phantom{0}-2\phantom{-0} & \phantom{-00}4.000000\phantom{-0}        &        
\phantom{}-127\phantom{-0} & \phantom{-0}5\phantom{-0} & \phantom{-0}24.993422\phantom{-0} \\\hline        
\phantom{0}-35\phantom{-0} & \phantom{0}-2\phantom{-0} & \phantom{-00}4.000000\phantom{-0}        &        
\phantom{}-143\phantom{-0} & \phantom{00}0\phantom{-0} & \phantom{00}-0.012195\phantom{-0} \\\hline        
\phantom{0}-52\phantom{-0} & \phantom{0}-4\phantom{-0} & \phantom{-0}16.000000\phantom{-0}        &        
\phantom{}-159\phantom{-0} & \phantom{00}0\phantom{-0} & \phantom{00}-0.097331\phantom{-0} \\\hline        
\phantom{0}-55\phantom{-0} & \phantom{-0}2\phantom{-0} & \phantom{-00}4.000000\phantom{-0}        &        
\phantom{}-164\phantom{-0} & \phantom{-0}2\phantom{-0} & \phantom{-00}3.933087\phantom{-0} \\\hline        
\phantom{0}-59\phantom{-0} & \phantom{0}-3\phantom{-0} & \phantom{-00}8.999999\phantom{-0}        &        
\phantom{}-168\phantom{-0} & \phantom{00}0\phantom{-0} & \phantom{-00}0.067253\phantom{-0} \\\hline        
\phantom{0}-67\phantom{-0} & \phantom{0}-2\phantom{-0} & \phantom{-00}4.000022\phantom{-0}        &        
\phantom{}-179\phantom{-0} & \phantom{00}0\phantom{-0} & \phantom{00}-0.062437\phantom{-0} \\\hline        
\phantom{0}-68\phantom{-0} & \phantom{-0}2\phantom{-0} & \phantom{-00}4.000001\phantom{-0}        &        
\phantom{}-183\phantom{-0} & \phantom{-0}4\phantom{-0} & \phantom{-0}16.008922\phantom{-0} \\\hline        
\phantom{0}-79\phantom{-0} & \phantom{-0}5\phantom{-0} & \phantom{-0}25.000105\phantom{-0}        &        
\phantom{}-184\phantom{-0} & \phantom{}-12\phantom{-0} & \phantom{-}144.143109\phantom{-0} \\\hline        
\phantom{0}-87\phantom{-0} & \phantom{-0}2\phantom{-0} & \phantom{-00}4.000124\phantom{-0}        &        
\phantom{}-187\phantom{-0} & \phantom{0}-2\phantom{-0} & \phantom{-00}3.453849\phantom{-0} \\\hline        
\end{tabular}
\end{table}

\begin{table}[p]
\caption{Data for the paramodular form of level 461, based on the Hasse-Weil $L$-series for the curve $y^{2} + y = -2x^{6} + 3x^{5} - 3x^{3} + x$. The constant ${C_{461}} = 12.0599439822$.}
\label{tbl:values461}
\begin{tabular}{|c|c|c||c|c|c|}
\hline & & & & & \\[-2ex]
$D$ & $A(D; {F_{461}})$ & $\frac{L({F_{461}},1/2,\chi_D)}{{C_{461}}}\,\abs{D}$
& $D$ & $A(D; {F_{461}})$ & $\frac{L({F_{461}},1/2,\chi_D)}{{C_{461}}}\,\abs{D}$
\\[0.8ex] \hline\hline
\phantom{00}-4\phantom{-0} & \phantom{-0}1\phantom{-0} & \phantom{-00}1.000000\phantom{-0}        &        
\phantom{}-103\phantom{-0} & \phantom{0}-6\phantom{-0} & \phantom{-0}35.997894\phantom{-0} \\\hline        
\phantom{0}-19\phantom{-0} & \phantom{0}-1\phantom{-0} & \phantom{-00}1.000000\phantom{-0}        &        
\phantom{}-104\phantom{-0} & \phantom{-0}2\phantom{-0} & \phantom{-00}4.000584\phantom{-0} \\\hline        
\phantom{0}-20\phantom{-0} & \phantom{-0}1\phantom{-0} & \phantom{-00}1.000000\phantom{-0}        &        
\phantom{}-107\phantom{-0} & \phantom{00}0\phantom{-0} & \phantom{00}-0.006352\phantom{-0} \\\hline        
\phantom{0}-23\phantom{-0} & \phantom{-0}1\phantom{-0} & \phantom{-00}1.000000\phantom{-0}        &        
\phantom{}-111\phantom{-0} & \phantom{0}-4\phantom{-0} & \phantom{-0}16.005971\phantom{-0} \\\hline        
\phantom{0}-24\phantom{-0} & \phantom{0}-2\phantom{-0} & \phantom{-00}4.000000\phantom{-0}        &        
\phantom{}-115\phantom{-0} & \phantom{-0}7\phantom{-0} & \phantom{-0}48.982236\phantom{-0} \\\hline        
\phantom{0}-39\phantom{-0} & \phantom{-0}2\phantom{-0} & \phantom{-00}4.000000\phantom{-0}        &        
\phantom{}-120\phantom{-0} & \phantom{-0}2\phantom{-0} & \phantom{-00}4.003979\phantom{-0} \\\hline        
\phantom{0}-43\phantom{-0} & \phantom{-0}4\phantom{-0} & \phantom{-0}16.000000\phantom{-0}        &        
\phantom{}-132\phantom{-0} & \phantom{0}-6\phantom{-0} & \phantom{-0}36.033483\phantom{-0} \\\hline        
\phantom{0}-56\phantom{-0} & \phantom{00}0\phantom{-0} & \phantom{-00}0.000001\phantom{-0}        &        
\phantom{}-139\phantom{-0} & \phantom{-}10\phantom{-0} & \phantom{-0}99.938609\phantom{-0} \\\hline        
\phantom{0}-59\phantom{-0} & \phantom{0}-2\phantom{-0} & \phantom{-00}3.999996\phantom{-0}        &        
\phantom{}-143\phantom{-0} & \phantom{00}0\phantom{-0} & \phantom{-00}0.022373\phantom{-0} \\\hline        
\phantom{0}-67\phantom{-0} & \phantom{-0}2\phantom{-0} & \phantom{-00}4.000039\phantom{-0}        &        
\phantom{}-151\phantom{-0} & \phantom{0}-7\phantom{-0} & \phantom{-0}48.911564\phantom{-0} \\\hline        
\phantom{0}-68\phantom{-0} & \phantom{0}-1\phantom{-0} & \phantom{-00}0.999934\phantom{-0}        &        
\phantom{}-163\phantom{-0} & \phantom{0}-3\phantom{-0} & \phantom{-00}9.016937\phantom{-0} \\\hline        
\phantom{0}-84\phantom{-0} & \phantom{-0}6\phantom{-0} & \phantom{-0}36.000642\phantom{-0}        &        
\phantom{}-164\phantom{-0} & \phantom{0}-4\phantom{-0} & \phantom{-0}15.940089\phantom{-0} \\\hline        
\phantom{0}-87\phantom{-0} & \phantom{0}-2\phantom{-0} & \phantom{-00}4.000088\phantom{-0}        &        
\phantom{}-167\phantom{-0} & \phantom{-0}1\phantom{-0} & \phantom{-00}0.940123\phantom{-0} \\\hline        
\phantom{0}-88\phantom{-0} & \phantom{0}-2\phantom{-0} & \phantom{-00}4.002180\phantom{-0}        &        
\phantom{}-191\phantom{-0} & \phantom{0}-1\phantom{-0} & \phantom{-00}1.102735\phantom{-0} \\\hline        
\phantom{0}-91\phantom{-0} & \phantom{0}-4\phantom{-0} & \phantom{-0}15.999578\phantom{-0}        &        
\phantom{}-195\phantom{-0} & \phantom{0}-2\phantom{-0} & \phantom{-00}3.513855\phantom{-0} \\\hline        
\phantom{0}-95\phantom{-0} & \phantom{-0}3\phantom{-0} & \phantom{-00}9.000945\phantom{-0}        &        
\phantom{}-199\phantom{-0} & \phantom{-0}5\phantom{-0} & \phantom{-0}24.577953\phantom{-0} \\\hline        
\end{tabular}
\end{table}

\begin{table}[p]
\caption{Data for the paramodular form of level 523, based on the Hasse-Weil $L$-series for the curve $y^{2} + x y = -x^{5} + 4x^{4} - 5x^{3} + x^{2} + x$. The constant ${C_{523}} = 6.8275178004$.}
\label{tbl:values523}
\begin{tabular}{|c|c|c||c|c|c|}
\hline & & & & & \\[-2ex]
$D$ & $A(D; {F_{523}})$ & $\frac{L({F_{523}},1/2,\chi_D)}{{C_{523}}}\,\abs{D}$
& $D$ & $A(D; {F_{523}})$ & $\frac{L({F_{523}},1/2,\chi_D)}{{C_{523}}}\,\abs{D}$
\\[0.8ex] \hline\hline
\phantom{00}-3\phantom{-0} & \phantom{0}-1\phantom{-0} & \phantom{-00}1.000000\phantom{-0}        &        
\phantom{}-103\phantom{-0} & \phantom{0}-5\phantom{-0} & \phantom{-0}24.997989\phantom{-0} \\\hline        
\phantom{00}-8\phantom{-0} & \phantom{0}-2\phantom{-0} & \phantom{-00}4.000000\phantom{-0}        &        
\phantom{}-104\phantom{-0} & \phantom{0}-4\phantom{-0} & \phantom{-0}16.010823\phantom{-0} \\\hline        
\phantom{0}-20\phantom{-0} & \phantom{-0}2\phantom{-0} & \phantom{-00}4.000000\phantom{-0}        &        
\phantom{}-115\phantom{-0} & \phantom{-0}6\phantom{-0} & \phantom{-0}35.990265\phantom{-0} \\\hline        
\phantom{0}-35\phantom{-0} & \phantom{0}-4\phantom{-0} & \phantom{-0}16.000000\phantom{-0}        &        
\phantom{}-120\phantom{-0} & \phantom{0}-4\phantom{-0} & \phantom{-0}16.007819\phantom{-0} \\\hline        
\phantom{0}-39\phantom{-0} & \phantom{-0}2\phantom{-0} & \phantom{-00}4.000000\phantom{-0}        &        
\phantom{}-123\phantom{-0} & \phantom{0}-2\phantom{-0} & \phantom{-00}4.017347\phantom{-0} \\\hline        
\phantom{0}-47\phantom{-0} & \phantom{0}-5\phantom{-0} & \phantom{-0}25.000000\phantom{-0}        &        
\phantom{}-127\phantom{-0} & \phantom{00}0\phantom{-0} & \phantom{-00}0.041002\phantom{-0} \\\hline        
\phantom{0}-51\phantom{-0} & \phantom{00}0\phantom{-0} & \phantom{00}-0.000000\phantom{-0}        &        
\phantom{}-132\phantom{-0} & \phantom{0}-4\phantom{-0} & \phantom{-0}16.040526\phantom{-0} \\\hline        
\phantom{0}-55\phantom{-0} & \phantom{0}-4\phantom{-0} & \phantom{-0}15.999997\phantom{-0}        &        
\phantom{}-136\phantom{-0} & \phantom{-0}4\phantom{-0} & \phantom{-0}15.950713\phantom{-0} \\\hline        
\phantom{0}-56\phantom{-0} & \phantom{00}0\phantom{-0} & \phantom{00}-0.000003\phantom{-0}        &        
\phantom{}-139\phantom{-0} & \phantom{0}-8\phantom{-0} & \phantom{-0}63.909162\phantom{-0} \\\hline        
\phantom{0}-59\phantom{-0} & \phantom{0}-2\phantom{-0} & \phantom{-00}4.000011\phantom{-0}        &        
\phantom{}-148\phantom{-0} & \phantom{00}0\phantom{-0} & \phantom{-00}0.064277\phantom{-0} \\\hline        
\phantom{0}-67\phantom{-0} & \phantom{-0}5\phantom{-0} & \phantom{-0}24.999910\phantom{-0}        &        
\phantom{}-152\phantom{-0} & \phantom{0}-8\phantom{-0} & \phantom{-0}64.090096\phantom{-0} \\\hline        
\phantom{0}-79\phantom{-0} & \phantom{00}0\phantom{-0} & \phantom{00}-0.001005\phantom{-0}        &        
\phantom{}-155\phantom{-0} & \phantom{00}0\phantom{-0} & \phantom{-00}0.137795\phantom{-0} \\\hline        
\phantom{0}-83\phantom{-0} & \phantom{-0}3\phantom{-0} & \phantom{-00}8.999507\phantom{-0}        &        
\phantom{}-159\phantom{-0} & \phantom{00}0\phantom{-0} & \phantom{-00}0.313322\phantom{-0} \\\hline        
\phantom{0}-84\phantom{-0} & \phantom{-0}4\phantom{-0} & \phantom{-0}15.999928\phantom{-0}        &        
\phantom{}-163\phantom{-0} & \phantom{-0}9\phantom{-0} & \phantom{-0}80.912682\phantom{-0} \\\hline        
\phantom{0}-87\phantom{-0} & \phantom{0}-8\phantom{-0} & \phantom{-0}64.001544\phantom{-0}        &        
\phantom{}-167\phantom{-0} & \phantom{0}-2\phantom{-0} & \phantom{-00}4.081331\phantom{-0} \\\hline        
\phantom{0}-88\phantom{-0} & \phantom{-0}4\phantom{-0} & \phantom{-0}16.000150\phantom{-0}        &        
\phantom{}-184\phantom{-0} & \phantom{-0}4\phantom{-0} & \phantom{-0}16.041808\phantom{-0} \\\hline        
\phantom{0}-95\phantom{-0} & \phantom{00}0\phantom{-0} & \phantom{-00}0.000162\phantom{-0}        &        
\phantom{}-199\phantom{-0} & unknown & \phantom{-00}2.233478\phantom{-0} \\\hline        
\end{tabular}
\end{table}

\begin{table}[p]
\caption{Data for the paramodular form of level 587 (in the plus space), based on the Hasse-Weil $L$-series for the curve $y^{2} = -3x^{6} + 18x^{4} + 6x^{3} + 9x^{2} - 54x + 57$. The constant ${C_{587}^+} = 15.8250549126$.}
\label{tbl:values587p}
\begin{tabular}{|c|c|c||c|c|c|}
\hline & & & & & \\[-2ex]
$D$ & $A(D; {F_{587}^+})$ & $\frac{L({F_{587}^+},1/2,\chi_D)}{{C_{587}^+}}\,\abs{D}$
& $D$ & $A(D; {F_{587}^+})$ & $\frac{L({F_{587}^+},1/2,\chi_D)}{{C_{587}^+}}\,\abs{D}$
\\[0.8ex] \hline\hline
\phantom{00}-8\phantom{-0} & \phantom{-0}1\phantom{-0} & \phantom{-00}1.000000\phantom{-0}        &        
\phantom{}-107\phantom{-0} & \phantom{-0}3\phantom{-0} & \phantom{-00}8.992246\phantom{-0} \\\hline        
\phantom{0}-11\phantom{-0} & \phantom{-0}1\phantom{-0} & \phantom{-00}1.000000\phantom{-0}        &        
\phantom{}-111\phantom{-0} & \phantom{00}0\phantom{-0} & \phantom{-00}0.001034\phantom{-0} \\\hline        
\phantom{0}-15\phantom{-0} & \phantom{0}-1\phantom{-0} & \phantom{-00}1.000000\phantom{-0}        &        
\phantom{}-123\phantom{-0} & \phantom{-0}2\phantom{-0} & \phantom{-00}3.967416\phantom{-0} \\\hline        
\phantom{0}-19\phantom{-0} & \phantom{-0}1\phantom{-0} & \phantom{-00}1.000000\phantom{-0}        &        
\phantom{}-127\phantom{-0} & \phantom{0}-1\phantom{-0} & \phantom{-00}1.021384\phantom{-0} \\\hline        
\phantom{0}-20\phantom{-0} & \phantom{-0}1\phantom{-0} & \phantom{-00}1.000000\phantom{-0}        &        
\phantom{}-131\phantom{-0} & \phantom{-0}3\phantom{-0} & \phantom{-00}8.952729\phantom{-0} \\\hline        
\phantom{0}-23\phantom{-0} & \phantom{-0}1\phantom{-0} & \phantom{-00}1.000000\phantom{-0}        &        
\phantom{}-132\phantom{-0} & \phantom{-0}5\phantom{-0} & \phantom{-0}24.978160\phantom{-0} \\\hline        
\phantom{0}-24\phantom{-0} & \phantom{-0}1\phantom{-0} & \phantom{-00}1.000000\phantom{-0}        &        
\phantom{}-136\phantom{-0} & \phantom{0}-2\phantom{-0} & \phantom{-00}4.004657\phantom{-0} \\\hline        
\phantom{0}-35\phantom{-0} & \phantom{-0}1\phantom{-0} & \phantom{-00}1.000000\phantom{-0}        &        
\phantom{}-139\phantom{-0} & \phantom{-0}1\phantom{-0} & \phantom{-00}0.975103\phantom{-0} \\\hline        
\phantom{0}-39\phantom{-0} & \phantom{0}-1\phantom{-0} & \phantom{-00}1.000000\phantom{-0}        &        
\phantom{}-148\phantom{-0} & \phantom{0}-8\phantom{-0} & \phantom{-0}63.992830\phantom{-0} \\\hline        
\phantom{0}-52\phantom{-0} & \phantom{-0}1\phantom{-0} & \phantom{-00}1.000000\phantom{-0}        &        
\phantom{}-155\phantom{-0} & \phantom{0}-8\phantom{-0} & \phantom{-0}64.026128\phantom{-0} \\\hline        
\phantom{0}-56\phantom{-0} & \phantom{-0}3\phantom{-0} & \phantom{-00}9.000000\phantom{-0}        &        
\phantom{}-164\phantom{-0} & \phantom{-0}4\phantom{-0} & \phantom{-0}15.945677\phantom{-0} \\\hline        
\phantom{0}-71\phantom{-0} & \phantom{-0}2\phantom{-0} & \phantom{-00}4.000060\phantom{-0}        &        
\phantom{}-168\phantom{-0} & \phantom{0}-1\phantom{-0} & \phantom{-00}1.251997\phantom{-0} \\\hline        
\phantom{0}-91\phantom{-0} & \phantom{0}-5\phantom{-0} & \phantom{-0}25.000542\phantom{-0}        &        
\phantom{}-183\phantom{-0} & \phantom{0}-5\phantom{-0} & \phantom{-0}25.081099\phantom{-0} \\\hline        
\phantom{}-103\phantom{-0} & \phantom{0}-3\phantom{-0} & \phantom{-00}8.997703\phantom{-0}        &        
\phantom{}-187\phantom{-0} & \phantom{0}-2\phantom{-0} & \phantom{-00}4.172874\phantom{-0} \\\hline        
\end{tabular}
\end{table}

\begin{table}[p]
\caption{Data for the paramodular form of level 587 (in the minus space), based on the Hasse-Weil $L$-series for the curve $y^{2} + \left(x^{3} + x + 1\right) y = -x^{3} - x^{2}$. The constant ${C_{587}^-} = 12.6406580054$.}
\label{tbl:values587m}
\begin{tabular}{|c|c||c|c|}
\hline & & & \\[-2ex]
$D$ & $\frac{L({F_{587}^-},1/2,\chi_D)}{{C_{587}^-}}\,\lvert{D}\rvert$
& $D$ & $\frac{L({F_{587}^-},1/2,\chi_D)}{{C_{587}^-}}\,\lvert{D}\rvert$
\\[0.8ex] \hline\hline
\phantom{00}-3\phantom{-0} & \phantom{-00}1.000000\phantom{-0}        &        
\phantom{0}-95\phantom{-0} & \phantom{-00}0.988953\phantom{-0} \\\hline        
\phantom{00}-4\phantom{-0} & \phantom{-00}1.000000\phantom{-0}        &        
\phantom{}-104\phantom{-0} & \phantom{-00}0.998569\phantom{-0} \\\hline        
\phantom{00}-7\phantom{-0} & \phantom{-00}1.000000\phantom{-0}        &        
\phantom{}-115\phantom{-0} & \phantom{-}528.955763\phantom{-0} \\\hline        
\phantom{0}-31\phantom{-0} & \phantom{-00}4.000000\phantom{-0}        &        
\phantom{}-116\phantom{-0} & \phantom{-00}0.987458\phantom{-0} \\\hline        
\phantom{0}-40\phantom{-0} & \phantom{-00}9.000000\phantom{-0}        &        
\phantom{}-119\phantom{-0} & \phantom{-00}0.022887\phantom{-0} \\\hline        
\phantom{0}-43\phantom{-0} & \phantom{-}144.000000\phantom{-0}        &        
\phantom{}-120\phantom{-0} & \phantom{-0}25.008747\phantom{-0} \\\hline        
\phantom{0}-47\phantom{-0} & \phantom{-00}1.000001\phantom{-0}        &        
\phantom{}-143\phantom{-0} & \phantom{-00}0.856318\phantom{-0} \\\hline        
\phantom{0}-51\phantom{-0} & \phantom{-00}4.000000\phantom{-0}        &        
\phantom{}-151\phantom{-0} & \phantom{-00}0.151989\phantom{-0} \\\hline        
\phantom{0}-55\phantom{-0} & \phantom{-00}9.000004\phantom{-0}        &        
\phantom{}-152\phantom{-0} & \phantom{-00}0.975679\phantom{-0} \\\hline        
\phantom{0}-59\phantom{-0} & \phantom{-0}16.000007\phantom{-0}        &        
\phantom{}-159\phantom{-0} & \phantom{-00}1.086667\phantom{-0} \\\hline        
\phantom{0}-67\phantom{-0} & \phantom{-}143.999732\phantom{-0}        &        
\phantom{}-163\phantom{-0} & \phantom{-}898.880486\phantom{-0} \\\hline        
\phantom{0}-68\phantom{-0} & \phantom{-00}4.000069\phantom{-0}        &        
\phantom{}-167\phantom{-0} & \phantom{-00}4.308248\phantom{-0} \\\hline        
\phantom{0}-79\phantom{-0} & \phantom{-00}0.000216\phantom{-0}        &        
\phantom{}-179\phantom{-0} & \phantom{-0}10.246464\phantom{-0} \\\hline        
\phantom{0}-83\phantom{-0} & \phantom{-00}4.000474\phantom{-0}        &        
\phantom{}-184\phantom{-0} & \phantom{-0}25.006321\phantom{-0} \\\hline        
\phantom{0}-84\phantom{-0} & \phantom{-00}0.999879\phantom{-0}        &        
\phantom{}-191\phantom{-0} & \phantom{-00}0.537613\phantom{-0} \\\hline        
\phantom{0}-87\phantom{-0} & \phantom{-00}9.000104\phantom{-0}        &        
\phantom{}-195\phantom{-0} & \phantom{-00}8.602606\phantom{-0} \\\hline        
\phantom{0}-88\phantom{-0} & \phantom{-00}0.999560\phantom{-0}        &        
\phantom{}-199\phantom{-0} & \phantom{-00}0.158427\phantom{-0} \\\hline        
\end{tabular}
\end{table}


\bibliography{bocherer}
\bibliographystyle{plain}

\end{document}